\documentclass[a4paper,leqno,10pt]{amsart}

\usepackage{epsf,graphicx,epsfig}
\usepackage{amscd}
\usepackage{amsmath,latexsym,amssymb,amsthm}
\usepackage[nospace,noadjust]{cite}
\usepackage{textcomp}
\usepackage{setspace,cite}
\usepackage{lscape,fancyhdr,fancybox}
\usepackage{stmaryrd}
\usepackage[all,cmtip]{xy}
\usepackage{tikz}
\usepackage{cancel}
\usetikzlibrary{shapes,arrows,decorations.markings}
\usepackage{graphicx}

\usepackage{color}
\usepackage{url}
\usepackage{enumerate}
\usepackage[mathscr]{euscript}

  \theoremstyle{plain}

\swapnumbers
    \newtheorem{thm}{Theorem}[section]
    \newtheorem{prop}[thm]{Proposition}
     
   \newtheorem{lemma}[thm]{Lemma}

    \newtheorem{subsec}[thm]{}
\theoremstyle{definition}
    \newtheorem{defn}[thm]{Definition}
        \newtheorem{remark}[thm]{Remark}
    \newtheorem{exam}[thm]{Example}

\theoremstyle{remark}

\title{}
\author{}
\date{}
\usepackage{amssymb}

\usepackage{hyperref}
\hypersetup{
	colorlinks,
	citecolor=blue,
	filecolor=black,
	linkcolor=blue,
	urlcolor=black
}

\begin{document}
\title{Cohomology and deformations of hom-dendriform algebras and coalgebras}

\author{Apurba Das}
\address{Department of Mathematics and Statistics,
Indian Institute of Technology, Kanpur 208016, Uttar Pradesh, India.}
\email{apurbadas348@gmail.com}

\curraddr{}
\email{}


\begin{abstract}
Hom-dendriform algebras are twisted analog of dendriform algebras and are splitting
of hom-associative algebras. In this paper, we define a cohomology and deformation for hom-dendriform algebras. We relate this cohomology with the Hochschild-type cohomology of hom-associative algebras. We also describe similar results for the twisted analog of dendriform coalgebras.
\end{abstract}

\noindent

\thispagestyle{empty}

\maketitle


\vspace{0.2cm}

\section{Introduction}

Hom-type algebras first arise in the work of Hartwig, Larsson and Silvestrov in their study of the deformations of the Witt and Virasoro algebras using $\sigma$-derivations \cite{hart-larson-sil}. They observed that the space of $\sigma$-derivations satisfy a variation of the Jacobi identity, twisted by $\sigma$. An algebra satisfying such an identity is called a hom-Lie algebra. Other types of algebras (such as associative, Leibniz, Poisson, Hopf,$ \ldots)$ twisted by homomorphisms have also been studied in the last few years. See \cite{amm-ej-makh, makh-sil} and references therein for more details about such structures.

In this paper, we deal with the twisted version of another type algebras, called dendriform algebras. These algebras were introduced by Loday as Koszul dual of associative dialgebras \cite{loday}. Free dendriform algebra over a vector space has been constructed by using planar binary trees.
Dendriform algebras also arise from Rota-Baxter operators on some associative algebra  \cite{aguiar}. Recently, an explicit cohomology theory for dendriform algebras has been introduced and the formal deformation theory for dendriform algebras (as well as coalgebras) has been studied in \cite{das3, das4}. The cohomology involves certain combinatorial maps. See the references therein for more details about dendriform structures.

The twisted version of dendriform structures, called hom-dendriform structures was introduced in \cite{makh}. These algebras can be thought of as splitting of hom-associative algebras. They also arise from Rota-Baxter operator on hom-associative algebras. In \cite{ma-zheng} the authors study hom-dendriform algebras from the point of view of monoidal categories.
Our aim in this paper is to twist the construction of \cite{das3} by a homomorphism to formulate a cohomology for hom-dendriform algebras. The cochain groups defining the cohomology inherits a structure of an operad with a multiplication. Hence by a result of Gerstenhaber and Voronov \cite{gers-voro}, the cohomology carries a Gerstenhaber structure. We show that there is a morphism from the cohomology of a hom-dendriform algebra to the Hochschild-type cohomology of the corresponding hom-associative algebra. We also study formal one-parameter deformation of a hom-dendriform algebra following the approach of Gerstenhaber \cite{gers}. Our results about deformation are similar to the standard ones. The vanishing of the second cohomology implies the rigidity of the structure and the vanishing of the third cohomology ensures that a finite order deformation can be extended to a deformation of the next order.

Finally, we consider a dual version of the above results. Namely, we consider hom-dendriform coalgebras, their cohomology, and deformations. We show that there is a morphism from the cohomology of a hom-dendriform coalgebra to the coHochschild-type cohomology of hom-associative coalgebra. In our knowledge, the coHochschild-type cohomology of a hom-associative coalgebra has not been mentioned before.

Throughout the paper, we freely use the Hopf algebra and coalgebra terminology introduced in \cite{dnr} and \cite{radford}.
All vector spaces, linear maps are over a field $\mathbb{K}$ of characteristic zero, unless specified otherwise.

\section{Hom-associative and hom-dendriform algebras}
In this section, we recall the hom-analog of associative and dendriform algebras. Our main references are \cite{amm-ej-makh, makh, makh-sil}.

A hom-vector space is a pair $(A, \alpha)$ consists of a vector space $A$ with a linear map $\alpha : A \rightarrow A$. Hom-associative or hom-dendriform structures are defined on hom-vector spaces, rather than vector spaces.

\begin{defn}
A hom-associative algebra is a hom-vector space $(A, \alpha)$ together with a bilinear map $\mu : A \times A \rightarrow A,~ (a, b) \mapsto a \cdot b$, that satisfies
\begin{align*}
\alpha (a) \cdot (b \cdot c) = (a \cdot b) \cdot \alpha (c), ~~~ \text{ for all } ~ a, b, c \in A.
\end{align*}
\end{defn}

A hom-associative algebra is called multiplicative if $\alpha (a \cdot b) = \alpha (a) \cdot \alpha (b)$. By a hom-associative algebra, we shall always mean a multiplicative one, unless specified otherwise.

In \cite{amm-ej-makh, makh-sil} the authors define a Hochschild-type cohomology of a hom-associative algebra and study the formal deformation theory for these type of algebras. Let $(A, \alpha, \mu)$ be a hom-associative algebra. For each $n \geq 1$, they define the group of $n$-cochains as
\begin{align*}
C^n_{\alpha, \mathrm{Ass}}(A, A):= \{ f : A^{\otimes n} \rightarrow A |~ \alpha \circ f (a_1, \ldots, a_n ) = f (\alpha (a_1), \ldots, \alpha (a_n)) \}
\end{align*} 
and the differential $\delta_{\alpha, \mathrm{Ass}} : C^n_{\alpha, \mathrm{Ass}} (A, A) \rightarrow C^{n+1}_{\alpha, \mathrm{Ass}} (A, A)$ is given by
\begin{align}\label{hom-ass-diff}
(\delta_{\alpha, \mathrm{Ass}} f)(a_1, \ldots, a_{n+1}) =~& \alpha^{n-1} (a_1) \cdot f (a_2, \ldots, a_{n+1}) \\
~&+ \sum_{i=1}^n (-1)^{i} f (   \alpha (a_1), \ldots, \alpha (a_{i-1}), a_i \cdot a_{i+1}, \alpha (a_{i+2}), \ldots, \alpha (a_{n+1}) ) \nonumber \\
~&+ (-1)^{n+1} f (a_1, \ldots, a_n) \cdot \alpha^{n-1} (a_{n+1}). \nonumber
\end{align}
for $a_1, \ldots, a_{n+1} \in A$. It has been shown in \cite{das2} that the above cochain groups can be endowed with a structure of an operad with a multiplication. 
Hence by a result of Gerstenhaber and Voronov \cite{gers-voro}, the cohomology inherits a Gerstenhaber structure. 
More precisely, the partial compositions of the operad are given by
\begin{align*}
 (f \bullet_i g ) (a_1, \ldots, a_{m+n-1}) =  f (\alpha^{n-1} a_1, \ldots, \alpha^{n-1} a_{i-1}, g(a_i, \ldots, a_{i+n-1}), \alpha^{n-1} a_{i+n}, \ldots, \alpha^{n-1} a_{m+n-1}),
\end{align*}
for $f \in C^m_{\alpha, \mathrm{Ass}} (A, A), ~g \in C^n_{\alpha, \mathrm{Ass}} (A, A)$ and $1 \leq i \leq m$. The multiplication $\pi \in C^2_{\alpha, \mathrm{Ass}} (A, A)$ is given by $\pi (a, b) = a \cdot b$. One observes that the differential induced from the above multiplication is exactly given by (\ref{hom-ass-diff}). See \cite{gers-voro, das2, lod-val} for details about operads. We will follow their convensions throughout.

\begin{defn}
A hom-dendriform algebra consists of a hom-vector space $(A, \alpha)$ together with linear maps $\prec, \succ : A \otimes A \rightarrow A$ satisfying the following identities:
\begin{align}
(a \prec b ) \prec \alpha (c) =~& \alpha (a) \prec ( b \prec c + b \succ c ), \label{dend-def-1}\\
(a \succ b ) \prec \alpha (c) =~& \alpha (a) \succ (b \prec c),\\
(a \prec b + a \succ b) \succ \alpha (c) =~&  \alpha (a) \succ (b \succ c), \label{dend-def-3}
\end{align}
for all $a, b, c \in A$.
\end{defn}

A hom-dendriform algebra as above is denoted by the quadruple $(A, \alpha, \prec, \succ )$. A hom-dendriform algebra is said to be multiplicative if $\alpha ( a \prec b ) = \alpha (a) \prec \alpha (b)$ and $\alpha (a \succ b ) = \alpha (a) \succ \alpha (b)$, for $a, b \in A$. When $\alpha = \text{id}_A$, one obtains the definition of a dendriform algebra \cite{loday}. In the rest of the paper, by a hom-dendriform algebra, we shall always mean a multiplicative hom-dendriform algebra unless otherwise stated.

\begin{exam}
Let $(A, \prec, \succ)$ be a dendriform algebra and $\alpha : A \rightarrow A$ be a morphism between dendriform algebras. Then the quadruple $(A, \alpha, \alpha \circ \prec,~ \alpha \circ \succ)$ is a hom-dendriform algebra.
\end{exam}

\begin{exam}\label{rota-h} (Rota-Baxter operator on hom-associative algebra) Let $(A, \alpha, \mu)$ be a hom-associative algebra. A linear map $R: A \rightarrow A$ is said to be a Rota-Baxter operator (of weight zero) on $A$ if $\alpha \circ R = R \circ \alpha$ and the following holds
\begin{align*}
R(a) \cdot R(b) = R(a \cdot R(b) + R(a) \cdot b ), ~~ \text{ for all } a, b \in A.
\end{align*}
Then it follows that the hom-vector space $(A, \alpha)$ together with the operations $a \prec b = a \cdot R(b)$  and $a \succ b = R(a) \cdot b$ forms a hom-dendriform algebra \cite{makh}.
\end{exam}

\begin{remark}
In \cite{liu-makh-men-pan} the authors define Rota-Baxter operator on bihom-associative algebras (associative algebras twisted by two homomorphisms) and show that such an operator induces a bihom-dendriform structure thus generalizing Example \ref{rota-h}.
\end{remark}

In the next, we generalize the above example of hom-dendriform algebra in the setting of $\mathcal{O}$-operators on hom-associative algebras.

\begin{exam}\label{o-h} ($\mathcal{O}$-operator on hom-associative algebra) Let $A = ( A, \mu, \alpha )$ be a hom-associative algebra and $(M, \beta)$ be a hom-vector space together with maps $\mu^l : A \otimes M \rightarrow M, ~ (a,m) \rightarrow am$ and $\mu^r : M \otimes A \rightarrow M, ~ (m, a) \rightarrow ma$ satisfying $\beta (am) = \alpha (a) \beta (m)$, $\beta (ma ) = \beta (m) \alpha(a)$ and the followings
\begin{align*}
(a \cdot b) \beta (m) = \alpha (a) (bm), \quad (am) \alpha (b) = \alpha (a) (mb), \quad (ma) \alpha (b) = \beta(m) (a \cdot b).
\end{align*}
Such a tuple $(M, \beta, \mu^l, \mu^r)$ is called a representation of $A$. An $\mathcal{O}$-operator on $A$ with respect to the above representation is a linear map $R: M \rightarrow A$ satisfying $\alpha \circ R = R \circ \beta$ and $R(m) \cdot R(n) = R (m R(n) + R(m)n),$ for all $m, n \in M$. In this case, $(M, \beta)$ together with operations $m \prec n = m R(n)$ and $m \succ n = R(m)n$ forms a hom-dendriform algebra.
\end{exam}

Like dendriform algebras are related to several other algebras, hom-dendriform algebras are related to hom analog of various algebras \cite{makh}. Here we mention the relation with hom-associative algebras and hom-Lie algebras.
Let $(A, \alpha, \prec, \succ )$ be a hom-dendriform algebra. Then the new operation $a \ast b$  defined by
\begin{align*}
a \ast b = a \prec b + a \succ b
\end{align*}
makes $(A, \alpha, \ast)$ into a hom-associative algebra. Therefore, a hom-dendriform algebra can be thought of as an splitting of a hom-associative algebra. Note that the corresponding hom-Lie bracket obtained from the skew-symmetrization of the hom-associative product is given by $[a, b ] = a \ast b - b \ast a$. On the other hand, a hom-dendriform algebra $(A, \alpha, \prec, \succ )$ also induces a left hom-preLie product $ a \diamond b = a \succ b -  b \prec a$, that is, the following holds
\begin{align*}
( a \diamond b ) \diamond \alpha (c) - \alpha (a) \diamond ( b \diamond c) = ( b \diamond a ) \diamond \alpha (c) - \alpha (b) \diamond ( a \diamond c).
\end{align*}
See \cite{makh} for details. The skew-symmetrization of a left hom-preLie product is a hom-Lie bracket and here it is given by
$[a, b]' = a \diamond b - b \diamond a$, for $a, b \in A$.
Observe that the hom-Lie brackets on $A$ induced from the hom-associative product and hom-preLie product are same.

\section{Cohomology}\label{sec-3}

Our aim in this section is to introduce a cohomology theory for hom-dendriform algebras as a twisted analog of the cohomology introduced in \cite{das3}. This cohomology can be thought of as a splitting of the cohomology of hom-associative algebras. We will show that the cochain groups defining the cohomology of a hom-dendriform algebra carries a structure of an operad with a multiplication. (This construction is the twisted analog of the construction defined for dendriform algebras.) Hence the cohomology inherits a Gerstenhaber structure. In the next section, we show that this cohomology governs the formal deformation of the hom-dendriform structure.

We recall certain combinatorial maps which are defined in \cite{das3}. Let $C_n$ be the set of first $n$ natural numbers. For convenience, we denote them by $\{ [1], \ldots, [n] \}.$ For each $m, n \geq 1$ and $1 \leq i \leq m$, there are maps $R_0 (m; 1, \ldots, \underbrace{n}_{i\text{-th}} , \ldots, 1 ) : C_{m+n-1} \rightarrow C_m$ and $R_i (m; 1, \ldots, \underbrace{n}_{i\text{-th}}, \ldots, 1 ) : C_{m+n-1} \rightarrow \mathbb{K}[C_n]$ which are given by
\begin{align*} R_0 (m; 1, \ldots, 1, n, 1, \ldots, 1) ([r]) ~=~
\begin{cases} [r] ~~~ &\text{ if } ~~ r \leq i-1 \\ [i] ~~~ &\text{ if } i \leq r \leq i +n -1 \\
[r -n + 1] ~~~ &\text{ if } i +n \leq r \leq m+n -1, \end{cases}
\end{align*}
\begin{align*} R_i (m; 1, \ldots, 1, n, 1, \ldots, 1) ([r]) ~=~
\begin{cases} [1] + [2] + \cdots + [n] ~~~ &\text{ if } ~~ r \leq i-1 \\ [r - (i-1)] ~~~ &\text{ if } i \leq r \leq i +n -1 \\
[1]+ [2] + \cdots + [n] ~~~ &\text{ if } i +n \leq r \leq m+n -1. \end{cases}
\end{align*}

Let $(A, \alpha, \prec, \succ)$ be a hom-dendriform algebra. We define the group of $n$-cochains by $C^0_{\alpha, \mathrm{Dend}} (A, A) := 0$ and
\begin{align*}
C^n_{\alpha, \mathrm{Dend}} (A, A) := \big\{ f \in \text{Hom}_\mathbb{K} ( \mathbb{K}[C_n] \otimes A^{\otimes n}, A ) |~ \alpha (   f ([r]; a_1, \ldots, a_n )) = f ([r]; \alpha (a_1), \ldots, \alpha (a_n) ),\\ \text{ for all } [r] \in C_n \big\}, 
\end{align*}
for $n \geq 1.$ Then the collection of spaces $\mathcal{O}(n) = C^n_{\alpha, \mathrm{Dend}} (A, A), n \geq 1$, inherits a structure of an operad with partial compositions 
$\bullet_i : \mathcal{O}(m) \otimes \mathcal{O}(n) \rightarrow \mathcal{O}(m+n-1)$ given by
\begin{align*}
&( f \bullet_i g) ([r]; a_1, \ldots, a_{m+n-1})\\
&= f \big( R_0 (m; 1, \ldots, n, \ldots, 1)[r];~ \alpha^{n-1} a_1, \ldots, \alpha^{n-1} a_{i-1}, g (R_i (m; 1, \ldots, n, \ldots, 1)[r]; a_i, \ldots, a_{i+n-1}),\\
& \hspace*{10cm} ~~ \alpha^{n-1} a_{i+n}, \ldots, \alpha^{n-1} a_{m+n-1} \big),
\end{align*}
for $f \in \mathcal{O}(m), ~ g \in \mathcal{O}(n),~ [r] \in C_{m+n-1}$ and $a_1, \ldots, a_{m+n-1} \in A$. 
The proof is a combination of \cite{das3} where the case $\alpha = \text{id}_A$ has been considered, and \cite{das2} where the linear map $\alpha$ has been used to twist the endomorphism operad.

Note that, in this case, there are operations
\begin{align}
\bullet :~& \mathcal{O}(m ) \otimes \mathcal{O}(n) \rightarrow \mathcal{O}(m+n-1),~~~ f \bullet g := \sum_{i=1}^m (-1)^{(i-1)(n-1)} ~f \bullet_i g,\\
 [~, ~]:~& \mathcal{O}(m ) \otimes \mathcal{O}(n) \rightarrow \mathcal{O}(m+n-1), ~~~ [f, g] := f \bullet g - (-1)^{(m-1)(n-1)} g \bullet f.
 \end{align}
Moreover, one can define an element 
$\pi \in \mathcal{O}(2) = C^2_{\alpha, \mathrm{Dend}} (A, A)$ by
\begin{align*}
\pi ([1]; a, b ) = a \prec b ~~~ \text{ and } ~~~ \pi ( [2]; a , b) = a \succ b.
\end{align*}
It is a straightforward calculation that
\begin{align*}
(\pi \bullet \pi ) ([r]; a, b, c) =
\begin{cases} (a \prec b) \prec \alpha (c) - \alpha(a) \prec (b \prec c + b \succ c) ~~~ &\text{ if} ~~ [r] = [1] \\
(a \succ b) \prec \alpha (c) - \alpha (a) \succ (b \prec c)  ~~~ &\text{ if } [r] = [2] \\
(a \prec b + a \succ b ) \prec \alpha (c) - \alpha (a) \succ ( b \succ c)  ~~~ &\text{ if } [r] = [3]. \end{cases}
\end{align*}
Thus, it follows from the definition of a hom-dendriform algebra that $\pi$ defines a multiplication on the above operad $(\{ \mathcal{O}(n) \}_{n \geq 1}, \bullet_i ).$ The multiplication $\pi \in \mathcal{O}(2)$ induces an associative product $\mathcal{O}(m) \otimes \mathcal{O}(n) \rightarrow \mathcal{O}(m+n)$ given by
\begin{align}
 f \cdot g = (-1)^m~ (\pi \bullet_2 g) \bullet_1 f
 \end{align}
and a degree $+1$ differential $\delta_\pi : \mathcal{O}(n) \rightarrow \mathcal{O}(n+1), ~f \mapsto \pi \bullet f - (-1)^{n-1} f \bullet \pi$. We also denote this differential by $\delta_{\alpha, \mathrm{Dend}}$ and it is explicitly given by
\begin{align*}
&(\delta_{\mathrm{Dend}} f) ([r]; a_1 , \ldots, a_{n+1}) \\
&=  \pi \big( R_0 (2;1,n) [r]; ~\alpha^{n-1} a_1, f (R_2 (2;1,n)[r]; a_2, \ldots, a_{n+1})   \big) \\
&+ \sum_{i=1}^n  (-1)^i~ f \big(  R_0 (n; 1, \ldots, 2, \ldots, 1)[r]; \alpha(a_1), \ldots, \alpha(a_{i-1}), \pi (R_i (1, \ldots, 2, \ldots, 1)[r]; a_i, a_{i+1}), \alpha(a_{i+2}), \ldots, \alpha(a_{n+1})   \big) \\
&+ (-1)^{n+1} ~\pi \big( R_0 (2; n, 1) [r];~ f (R_1 (2;n,1)[r]; a_1, \ldots, a_n), \alpha^{n-1}(a_{n+1})   \big),
\end{align*}
for $f \in C^n_{\alpha, \mathrm{Dend}}(A, A), ~[r] \in C_{n+1}$ and $a_1, \ldots, a_{n+1} \in A.$ The corresponding cohomology groups are denoted by $H^n_{\alpha, \mathrm{Dend}}(A, A)$, for $n \geq 1$. When $\alpha = \mathrm{id}_A$, the above cohomology coincides with the one constructed in \cite{das3}.

\medskip

Hence it follows from Gerstenhaber and Voronov \cite{gers-voro} that the cochain groups of a hom-dendriform algebra inherit a homotopy $G$-algebra structure. As a consequence, the cohomology carries a Gerstenhaber structure.

It follows from the definition that
\begin{align*}
H^1_{\alpha, \text{Dend}} = \big\{ f : A \rightarrow A |~ \alpha \circ f = f \circ \alpha ~~~ \text{ and }
~&f ( a \prec b) = a \prec f (b) + f(a) \prec b,\\
~&f ( a \succ b) = a \succ f (b) + f(a) \succ b \big\}.
\end{align*}
Thus, $H^1_{\alpha, \text{Dend}} (A, A)$ is the space of all derivations for the products $\prec, \succ$ which commute with $\alpha$.  In the next section, we interpret the second cohomology group $H^2_{\alpha, \text{Dend}} (A, A)$ as the equivalence classes of infinitesimal deformations of $A$.

\begin{remark}
One may also define cohomology of a hom-dendriform algebra with coefficients in a suitable representation. In such a case, the second cohomology with coefficients in a representation can also be represented by the equivalence classes of abelian extensions. See \cite{das3} for the case of dendriform algebras, that is, when $\alpha = \text{id}_A$.
\end{remark}

In the next, we relate the cohomology of a hom-dendriform algebra (as defined above) with the cohomology of the corresponding hom-associative algebra. More precisely, we have the following whose proof is similar to \cite{das3}.

\begin{thm}
Let $(A, \alpha, \prec, \succ)$ be a hom-dendriform algebra with the corresponding hom-associative algebra $(A, \alpha, \star)$. Then the map
\begin{align*}
S : C^n_{\alpha, \mathrm{Dend}} (A, A) \rightarrow C^n_{\alpha , \mathrm{Ass}}(A, A), ~ f \mapsto f_{[1]} + \cdots + f_{[n]}, ~~~ \text{ for } n \geq 1
\end{align*}
defines a morphism between operads which preserve the respective multiplications. Hence the induced map $S : H^n_{\alpha, \mathrm{Dend}} (A, A) \rightarrow H^n_{\alpha, \mathrm{Ass}} (A, A)$ on cohomology is a morphism between Gerstenhaber algebras.
\end{thm}

\section{Formal deformations}

Our aim in this section is to study the formal deformation theory for hom-dendriform algebras along the line of Gerstenhaber \cite{gers}. We show that the cohomology of hom-dendriform algebra governs the corresponding deformation.

Let $(A, \alpha, \prec, \succ)$ be a hom-dendriform algebra. Consider the space $A[[t]]$ of formal power series in $t$ with coefficients in $A$. Then $A[[t]]$ is a $\mathbb{K}[[t]]$-module and $A[[t]] \cong A \otimes_\mathbb{K} \mathbb{K}[[t]]$ when $A$ is finite dimensional.

\begin{defn}
A formal deformation of a hom-dendriform algebra $(A, \alpha, \prec, \succ)$ consists of formal power series $\prec_t = \sum_{i=0}^\infty \prec_i t^i$ and $\succ_t = \sum_{i=0}^\infty \succ_i t^i$ of bilinear maps on $A$ (with $\prec_0 = \prec$ and $\succ_0 = \succ)$ such that $(A[[t]], \alpha, \prec_t, \succ_t)$ is a hom-dendriform algebra over $\mathbb{K}[[t]].$
\end{defn}

Thus, if $(\prec_t, \succ_t)$ is a deformation, then for each $n \geq 0$ and $a, b, c \in A$, one must have
\begin{align}
\sum_{i+j = n} (a \prec_i b ) \prec_j \alpha (c) =~& \sum_{i+j = n} \alpha (a) \prec_i ( b \prec_j c + b \succ_j c ), \label{defr-1}\\
\sum_{i+j = n} (a \succ_i b ) \prec_j \alpha (c) =~& \sum_{i+j = n} \alpha (a) \succ_i (b \prec_j c), \label{defr-2}\\
\sum_{i+j = n} (a \prec_i b + a \succ_i b) \succ_j \alpha (c) =~& \sum_{i+j = n} \alpha (a) \succ_i (b \succ_j c). \label{defr-3}
\end{align}
The identities (\ref{defr-1})-(\ref{defr-3}) are called deformation equations. For each $i \geq 0$, we define $\pi_i \in \mathcal{O}(2) = C^2_{\alpha, \mathrm{dend}} (A, A)$ by
\begin{align*}
\pi_i ([1]; a, b) = a \prec_i b ~~~~ \text{ and } ~~~~ \pi_i ([2]; a, b) = a \succ_i b.
\end{align*}
Then the deformation equations can be simply expressed as
\begin{align}
\sum_{i+j = n} \pi_i \bullet \pi_j = 0, ~~~ \text{ for } n \geq 0.
\end{align}
The above identity automatically holds for $n = 0$ as $\pi_0 = \pi$ defines a multiplication on the operad. For $n = 1$, we have $ \pi \bullet \pi_1 + \pi_1 \bullet \pi = 0$ which implies that $\delta_{\alpha, \text{Dend}} (\pi_1) = 0$. Hence $\pi_1 \in C^2_{\alpha, \mathrm{Dend}} (A, A)$ defines a $2$-cocycle in the hom-dendriform algebra cohomology of $A$. It is called the infinitesimal of the deformation.

\begin{defn}
Two deformations $(\prec_t, \succ_t)$ and $(\prec_t', \succ_t')$ of a hom-dendriform algebra $(A, \alpha, \prec, \succ)$ are said to be equivalent if there is a formal isomorphism $\Phi_t = \sum_{i=0}^\infty \Phi_i t^i : A[[t]] \rightarrow A[[t]]$ with each $\Phi_i \in \text{Hom} (A, A)$ that commute with $\alpha$ and $\Phi_0 = \text{id}_A$ such that
\begin{align*}
\Phi_t ( a \prec_t b ) = \Phi_t (a) \prec_t' \Phi_t (b)  ~~~~ \text{ ~~and ~~} ~~~~  \Phi_t ( a \succ_t b ) = \Phi_t (a) \succ_t' \Phi_t (b). 
\end{align*}
The above condition of equivalence can be expressed as
\begin{align*}
\sum_{i+j = n} \Phi_i \bullet \pi_j ([r]; a, b) = \sum_{i+j+k = n} \pi_k' ([r]; \Phi_i (a), \Phi_j (b)), ~~~ \text{ for all } n \geq 0.
\end{align*}
\end{defn}
This always holds for $n = 0$ (as $\prec_0 = \prec_0' = \prec$, $\succ_0 = \succ_0' = \succ$ and $\Phi_0 = \text{id}_A$), whereas for $n =1$, we obtain
\begin{align*}
\pi_1 - \pi_1' = \pi \bullet \Phi_1 - \Phi_1 \bullet \pi = \delta_{\alpha, \text{Dend}} (\Phi_1).
\end{align*}
This shows that equivalent deformations have cohomologous infinitesimals, hence, they corresponds to same cohomology class in $H^2_{\alpha, \mathrm{Dend}}(A, A).$

To obtain a one-to-one correspondence between the second cohomology group and equivalence classes of certain type deformations, we have to use the following truncated version of formal deformations.

\begin{defn}
An infinitesimal deformation of a hom-dendriform algebra $A = ( A, \alpha, \prec, \succ)$ is a deformation of $A$ over $\mathbb{K}[[t]]/ (t^2)$ (the local Artinian ring of dual numbers).
\end{defn}

In other words, an infinitesimal deformation of $A$ is given by a pair $(\prec_t, \succ_t)$ in which $\prec_t = \prec + \prec_1 t$ and $\succ_t = \succ + \succ_1 t$ such that $\pi_1 = (\prec_1, \succ_1)$ defines a $2$-cocycle in the cohomology of $A$. More precisely, we have the following.

\begin{prop}
There is a one-to-one correspondence between the space of equivalence classes of infinitesimal deformations and the second cohomology group $H^2_{\alpha, \mathrm{Dend}} (A, A).$
\end{prop}

\begin{proof}
It is already shown that if two infinitesimal deformations $(\prec_t =~ \prec + \prec_1 t , ~ \succ_t = ~ \succ + \succ_1 t)$ and $(\prec_t' =~ \prec + \prec_1' t , ~ \succ_t' =~  \succ + \succ_1' t)$ are equivalent, then the $2$-cocycles $\pi_1 = ( \prec_1, \succ_1)$ and $\pi_1' = ( \prec_1', \succ_1')$ are cohomologous. Hence, the map
\begin{align*}
\text{infinitesimal deformations}/\sim  ~~\longrightarrow~ H^2_{\alpha, \mathrm{Dend}} ( A, A)
\end{align*}
defined by $[(\prec_t, \succ_t)] \mapsto [\pi_1]$ is well defined. This map turns out to be bijective with the inverse given as follows. For any $2$-cocycle $\pi_1 = (\prec_1, \succ_1) \in C^2_{\alpha, \mathrm{Dend}} (A, A)$, the pair  $(\prec_t =~ \prec + \prec_1 t , ~ \succ_t =~  \succ + \succ_1 t)$ defines an infinitesimal deformation. If $\pi_1' = (\prec_1', \succ_1') \in C^2_{\alpha, \mathrm{Dend}} (A, A)$ is another $2$-cocycle cohomologous to $\pi_1$, then we have $\pi_1 - \pi_1' = \delta_{\alpha, \mathrm{Dend}} (\Phi_1)$, for some $\Phi_1 \in C^1_{\alpha, \mathrm{Dend}} (A, A) = \mathrm{Hom}( \mathbb{K}[C_1] \otimes A , A ) \simeq \mathrm{Hom} (A, A)$. In such a case, $\Phi_t = \text{id}_A + \Phi_1 t $ defines an equivalence between infinitesimal deformations $(\prec_t =~ \prec + \prec_1 t , ~ \succ_t =~  \succ + \succ_1 t)$ and $(\prec_t' =~ \prec + \prec_1' t , ~ \succ_t' =~  \succ + \succ_1' t)$. Hence the inverse map is also well defined.
\end{proof}

\begin{defn}
A deformation $(\prec_t, \succ_t)$ is said to be trivial if it is equivalent to the deformation $(\prec_t' = \prec,~ \succ_t' = \succ).$
\end{defn}

\begin{lemma}\label{non-triv-lemma}
Let $(\prec_t, \succ_t)$ be a non-trivial deformation of a hom-dendriform algebra $A$. Then it is equivalent to some deformation $(\prec_t', \succ_t')$ in which $\prec_t' =~ \prec + \sum_{i \geq p} \prec_i t^i$ and $\succ_t' =~ \succ + \sum_{i \geq p} \succ_i t^i$, where the first non-zero term $\pi_p = ( \prec_p, \succ_p )$ is a $2$-cocycle but not a coboundary.
\end{lemma}

\begin{proof}
Let $(\prec_t, \succ_t)$ be a non-trivial deformation of a hom-dendriform algebra such that $\pi_1 := (\prec_1, \succ_1) = 0, \ldots, \pi_{n-1} := (\prec_{n-1}, \succ_{n-1}) = 0$ and $\pi_n := (\prec_n, \succ_n)$ is the first non-zero term. Then it has been shown that $\pi_n = (\prec_n, \succ_n)$ is a $2$-cocycle. If $\pi_n$ is not a $2$-coboundary, we are done. If $\pi_n $ is a $2$-coboundary, say $\pi_n = - \delta_{\alpha, \mathrm{Dend}} (\Phi_n)$, for some $\Phi_n \in C^1_{\alpha, \mathrm{Dend}} (A, A) = \mathrm{Hom} (A, A)$, then setting $\Phi_t = \text{id}_A + t^n \Phi_n$. We define $\prec_t' = \Phi_t^{-1} \circ \prec_t \circ \Phi_t$ and $\succ_t' = \Phi_t^{-1} \circ \succ_t \circ \Phi_t$. Then $(\prec_t', \succ_t')$ defines a formal deformation of the form
\begin{align*}
\prec_t' =~ \prec +~ t^{n+1} \prec_{n+1}' + \cdots \quad \mathrm{ and } \quad \succ_t' =~ \succ +~ t^{n+1} \succ_{n+1}' + \cdots .
\end{align*}
Thus, it follows that $\pi_{n+1}' = (\prec_{n+1}', \succ_{n+1}')$ is a $2$-cocycle in $C^2_{\alpha, \mathrm{Dend}} (A, A)$. If it is not a coboundary, we are done. If it is a coboundary, we can apply the same method again. In this way, we get a required type of equivalent deformation.
\end{proof}

As a consequence, we get the following.
\begin{thm}\label{h2}
If $H^2_{\alpha, \mathrm{Dend}} (A, A) = 0$ then every deformation of $A$ is equivalent to a trivial deformation.
\end{thm}

\begin{proof}
Let $(\prec_t, \succ_t)$ be a formal deformation of $A$. If it is a trivial deformation, it is equivalent to itself.  On the other hand, if it is non-trivial, then by Lemma \ref{non-triv-lemma} and the fact that $H^2_{\alpha, \mathrm{Dend}} (A, A) = 0$, we have $(\prec_t, \succ_t)$ is equivalent to $(\prec_t' =~ \prec, \succ_t' =~ \succ)$. Hence the proof. 
\end{proof}

A hom-dendriform algebra $A$ is said to be rigid if every deformation of $A$ is equivalent to a trivial deformation. It follows from Theorem \ref{h2} that $H^2 = 0$ is a sufficient condition for the rigidity of a hom-dendriform algebra.

\subsection{Extensions of finite order deformation}

A deformation $(\prec_t, \succ_t)$ of a hom-dendriform algebra $A$ is said to be of order $n$ if 
$\prec_t$ and $\succ_t$ are of the form $\prec_t = \sum_{i=0}^n \prec_i t^i$ and $\succ_t =  \sum_{i=0}^n \succ_i t^i$. Here we discuss the problem of extension of a deformation of order $n$ to a deformation of next order.

Suppose there is an element $\pi_{n+1} = (\prec_{n+1}, \succ_{n+1})$ such that $(\overline{\prec}_t = \prec_t + \prec_{n+1} t^{n+1} ,~ \overline{\succ}_t = \succ_t + \succ_{n+1} t^{n+1})$ is a deformation of order $n+1$. Therefore, one additional deformation equation need to be satisfy 
\begin{align*}
\sum_{i+j = n+1} \pi_i \bullet \pi_j = 0.
\end{align*}
This is equivalent to
\begin{align*}
 \delta_{\alpha, \text{Dend}} ( \pi_{n+1} ) = - \sum_{i+j = n+1, i, j \geq 1} \pi_i \bullet \pi_j.
\end{align*}
The right hand side of the above equation is called the obstruction to extend the deformation. Thus, if an extension is possible, the obstruction is always given by a coboundary. However, in any case, we have the following.
\begin{lemma}
The obstruction is a $3$-cocycle in the hom-dendriform algebra cohomology of $A$, i.e
\begin{align*}
\delta_{\alpha, \mathrm{Dend}} (- \sum_{i+j = n+1, i, j \geq 1} \pi_i \bullet \pi_j) = 0
\end{align*}
\end{lemma}

\begin{proof}
For any $\pi, \pi' \in C^2_{\alpha, \text{Dend}} (A, A)$, it is easy to see that
\begin{align*}
\delta_{\alpha, \mathrm{Dend}} (\pi \bullet \pi') =  \pi \bullet \delta_{\alpha, \mathrm{Dend}} (\pi') - \delta_{\alpha, \mathrm{Dend}} (\pi) \bullet \pi' ~+~ \pi' \cdot \pi -~ \pi \cdot \pi'.
\end{align*}
Therefore,
\begin{align*}
\delta_{\alpha, \mathrm{Dend}} \big(- \sum_{i+j =n+1, i, j \geq 1} \pi_i \bullet \pi_{j} \big) =~& - \sum_{i+j =n+1, i, j \geq 1} \big( \pi_i \bullet \delta_{\alpha, \mathrm{Dend}} (\pi_j) - \delta_{\alpha, \mathrm{Dend}} (\pi_i) \bullet \pi_j \big) \\
=~& \sum_{p+q+r = n+1, p, q, r \geq 1} \big(  \pi_p \bullet (\pi_q \bullet \pi_r) - (\pi_p \bullet \pi_q) \bullet \pi_r \big)\\ =~& \sum_{p+q+r = n+1, p, q, r \geq 1} A_{p, q, r}    \qquad \mathrm{(say)}.
\end{align*}
The product $\circ$ is not associative, however, they satisfy the pre-Lie identities  \cite{gers}. This in particular implies that $A_{p, q, r} = 0$ whenever $q = r$. Finally, if $q \neq r$ then $A_{p, q, r} + A_{p, r, q} = 0$ by the pre-Lie identities. Hence we have $\sum_{p+q+r = n+1, p, q, r \geq 1} A_{p, q, r}   = 0$.
\end{proof}

Thus, we have the following.
\begin{thm}
If $H^3_{\alpha, \mathrm{Dend}} (A, A) = 0$ then every finite order deformation of $A$ extends to a deformation of next order.
\end{thm}

\section{Hom-dendriform coalgebras and deformations}

In this section, we consider the dual picture of the results as described in previous sections. Namely, we consider hom-dendriform coalgebras and study their deformations via a cohomology theory.

\begin{defn}
A hom-dendriform coalgebra is a hom-vector space $(C, \alpha )$ together with two linear maps $\triangle_\prec, \triangle_\succ : C \rightarrow C \otimes C$ satisfying the following identities
\begin{align}
(\triangle_\prec \otimes  \alpha) \circ \triangle_\prec =~& (\alpha \otimes (\triangle_\prec + \triangle_\succ)) \circ \triangle_\prec, \label{c1}\\
(\triangle_\succ \otimes \alpha) \circ \triangle_\prec =~& (\alpha \otimes \triangle_\prec ) \circ \triangle_\succ, \label{c2}\\
((\triangle_\prec + \triangle_\succ ) \otimes \alpha ) \circ \triangle_\succ =~&  (\alpha \otimes \triangle_\succ ) \circ \triangle_\succ . \label{c3}
\end{align}
\end{defn}

Note that the above three identities are dual to the identities (\ref{dend-def-1})-(\ref{dend-def-3}). A hom-dendriform coalgebra as above is said to be multiplicative if $(\alpha \otimes \alpha ) \circ \triangle_\prec = \triangle_\prec \circ \alpha$ and $(\alpha \otimes \alpha ) \circ \triangle_\succ = \triangle_\succ \circ \alpha$. In either case, when $\alpha = \text{id}_C$, we obtain dendriform coalgebras \cite{jian, das3}.

A hom-dendriform coalgebra is a splitting of an associative coalgebra in the sense that if $(C, \alpha, \triangle_\prec, \triangle_\succ)$ is a hom-dendriform coalgebra, then $(C, \alpha, \triangle_\prec + \triangle_\succ)$ is a hom-associative coalgebra. See \cite{makh-sil3} for details about hom-associative coalgebras.

Any hom-associative coalgebra is a hom-dendriform coalgebra with either $\triangle_\prec = 0$ or $\triangle_\succ = 0$. One can construct a hom-dendriform coalgebra out of a dendriform coalgebra and a morphism of it. Similarly, dual to the Examples \ref{rota-h} and \ref{o-h}, one can define Rota-Baxter operator and $\mathcal{O}$-operator on hom-associative coalgebras which induce hom-dendriform coalgebras.

In the next, we define a cohomology for multiplicative hom-dendriform coalgebras. This cohomology can be thought of as a splitting of the Cartier (coHochschild-type) cohomology of hom-associative coalgebras. We start with the following twisted analog of the coendomorphism operad.

Let $(C, \alpha)$ be a hom-vector space. We define $C^0_{\alpha, \mathrm{coAss}} (C, C) = 0$ and 
\begin{align*}
C^n_{\alpha, \mathrm{coAss}} (C, C) = \{ \sigma : C \rightarrow C^{\otimes n} |~ \alpha^{\otimes n} \circ \sigma = \sigma \circ \alpha \}, ~~~ \text{ for } n \geq 1.
\end{align*}

\begin{prop}\label{coendo-op}
The collection of vector spaces $\{ C^n_{\alpha, \mathrm{coAss}} (C, C) \}_{n \geq 1}$ forms an operad with partial compositions
\begin{align*}
\sigma \bullet_i \tau  = (({\alpha^{n-1}})^{\otimes (i-1)} \otimes \tau \otimes ({\alpha^{n-1}})^{\otimes (m-i)}) \circ \sigma,
\end{align*}
for $\sigma \in C^m_{\alpha, \mathrm{coAss}} (C, C),~\tau \in C^n_{\alpha, \mathrm{coAss}} (C, C)$ and $[r] \in C_{[m+n-1]}$.
\end{prop}

The proof of the above proposition is dual to the proof of \cite[Proposition 3.2]{das2}. When $\alpha = \text{id}_C,$ one obtain the coendomorphism operad associated to the vector space $C$. Note that
a multiplication on the operad of Proposition \ref{coendo-op} is given by an element $\triangle \in C^2_{\alpha, \mathrm{coAss}} (C, C)$ such that $\triangle \bullet_1 \triangle = \triangle \bullet_2 \triangle$. In other words, $\triangle \in C^2_{\alpha, \mathrm{coAss}} (C, C)$ satisfies
\begin{align*}
(\triangle \otimes \alpha ) \circ \triangle = (\alpha \otimes \triangle) \circ \triangle.
\end{align*}
Therefore, $\triangle$ defines a multiplicative hom-associative coalgebra on $(C, \alpha).$ The induced differential $\delta_\triangle : C^n_{\alpha, \mathrm{coAss}} (C, C) \rightarrow C^{n+1}_{\alpha, \mathrm{coAss}} (C, C)$ (also denoted by $\delta_{\alpha, \mathrm{coAss}}$) is given by
\begin{align*}
\delta_{\alpha, \mathrm{coAss}} (f) = \triangle \bullet f - (-1)^{n+1} f \bullet \triangle, ~~~ \text{ for } f \in  C^n_{\alpha, \mathrm{coAss}} (C, C).
\end{align*}
This cohomology is called the Cartier (coHochschild-type) cohomology of the multiplicative hom-associative coalgebra $(C, \alpha, \triangle)$. Since this cohomology is induced from an operad with a multiplication, it follows from \cite{gers-voro} that this cohomology inherits a Gerstenhaber structure.

\medskip

Next, we consider a new operad associated to a hom-vector space dual to the operad given in Section \ref{sec-3}. Let $(C, \alpha )$
 be a hom-vector space. Define $C^0_{\alpha, \text{coDend}} (C, C) = 0$ and 
 \begin{align*}
 C^n_{\alpha, \text{coDend}} (C, C) = \{ \sigma : \mathbb{K}[C_n] \otimes C \rightarrow C^{\otimes n} |~  \alpha^{\otimes n} \circ \sigma ([r]; a) = \sigma ([r]; \alpha (a) ), \forall ~[r] \in C_n \}.
 \end{align*}
 
 Then we have the following.
 
 \begin{thm}
 The collections of vector spaces $\{ C^n_{\alpha, \mathrm{coDend}} (C,C) \}_{n \geq 1}$ forms an operad with partial compositions 
\begin{align*}
& (\sigma \bullet_i \tau) ([r]; ~\_~) \\
~& = ((\alpha^{n-1})^{\otimes (i-1)} \otimes g (R_i (m; 1, \ldots, n, \ldots, 1)[r]; ~\_~) \otimes (\alpha^{n-1})^{\otimes (m-i)}) \circ f (R_0 (m; 1, \ldots, n, \ldots, 1)[r]; ~\_~)
\end{align*}
 for $\sigma \in C^m_{ \alpha, \mathrm{coDend}} (C, C),~ \tau \in C^n_{\alpha, \mathrm{coDend}} (C, C), ~[r] \in C_{m+n-1},$
and the identity element $\mathrm{id} \in  C^1_{\alpha, \mathrm{coDend}} (C,C)$ given by $\mathrm{id} ([1]; c) = c,$ for all $c \in C.$ 
 
 Moreover, the collection of maps $\{ \Phi_n : C^n_{\alpha, \mathrm{coDend}} (C, C) \rightarrow C^n_{ \alpha, \mathrm{coAss}} (C, C) \}_{n \geq 1}$ given by
 \begin{align*}
 \Phi_n (\sigma ) = \sigma ([1]; ~) + \cdots + \sigma ([n]; ~ ), ~~~~ \text{ for } \sigma \in C^n_{\alpha, \mathrm{coDend}} (C, C)
 \end{align*}
 is a morphism between operads.
 \end{thm}
 
 The proof of the above theorem is along the same line of \cite{das3, das4}. Hence we will not repeat it here. Note that a multiplication in the operad $\{ C^n_{\alpha, \mathrm{coDend}} (C,C) , \bullet_i \}_{n \geq 1}$ is given by an element $\triangle \in C^2_{\alpha, \mathrm{coDend}} (C, C)$ satisfying $\triangle \bullet_1 \triangle = \triangle \bullet_2 \triangle$. The element $\triangle$ is equivalent to two maps $\triangle_\prec, \triangle_\succ : C \rightarrow C \otimes C$ given by
 \begin{align*}
 \triangle_\prec = \triangle ([1]; ~)   \quad \text{ and } \quad \triangle_\succ = \triangle ([2]; ~).
 \end{align*}
 The condition $\triangle \bullet_1 \triangle = \triangle \bullet_2 \triangle$ is equivalent to the fact that $(\triangle_\prec, \triangle_\succ)$ satisfy the identities (\ref{c1})-(\ref{c3}). Therefore, a multiplication is given by a multiplicative hom-dendriform structure on $(C, \alpha)$. We define a differential $\delta_{\alpha, \text{coDend}} : C^n_{\alpha, \text{coDend}} (C,C) \rightarrow C^{n+1}_{\alpha, \text{coDend}} (C,C)$ to be the one induced from the multiplication $\triangle$ on this operad, i.e. $\delta_{\alpha, \text{coDend}} (\sigma) = \triangle \bullet \sigma - (-1)^{n-1} \sigma \bullet \triangle$, for $\sigma \in C^n_{\alpha, \text{coDend}}(C, C)$.
  Explicitly, it is given by
  \begin{align*}
( \delta_{\alpha, \mathrm{coDend}} (\sigma) )& ([r] \otimes c) \\
 =~& (\alpha^{n-1} \otimes \sigma (R_2 (2;1,n)[r];~ \_~ )) \circ \triangle (R_0 (2;1, n)[r]; c) \nonumber \\
+~& \sum_{i=1}^n  (-1)^i ~({\alpha}^{\otimes (i-1)} \otimes \triangle_{R_i (n; 1, \ldots, 2, \ldots, 1)[r]} \otimes {\alpha}^{\otimes (n-i)}) \circ \sigma (R_0 (n;1, \ldots, 2, \ldots, 1)[r]; c) \nonumber \\
+~& (-1)^{n+1}~ (\sigma (R_1 (2;n,1)[r]; ~\_ ~) \otimes \alpha^{n-1} ) \circ  \triangle (R_0 (2;n,1)[r]; c), \nonumber
\end{align*}
for $\sigma \in C^n_{\alpha, \text{coDend}}(C, C)$, $[r] \in C_{n+1}$ and $c \in C$.

We denote the corresponding cohomology by $H^*_{\alpha, \mathrm{coDend}} (C, C)$. 
 When $\alpha = \text{id}_C$, this cohomology coincides with the one constructed in \cite{das4}.
Note that, since the above cohomology is induced from an operad with a multiplication, the cohomology inherits a Gerstenhaber structure.
 
 \subsection{Deformation}
 Here we study deformations of multiplicative hom-dendriform coalgebras. This is dual to the deformation of hom-dendriform algebras. Therefore, we will only state the results.
 
 A deformation of a multiplicative hom-dendriform coalgebra $(C, \alpha, \triangle_\prec, \triangle_\succ)$ is given by two formal sums $\triangle_{\prec, t} = \sum_{i \geq 0} t^i \triangle_{\prec, i} $ and $\triangle_{\succ, t} = \sum_{i \geq 0} t^i \triangle_{\succ, i}$ (with $\triangle_{\prec, 0} = \triangle_\prec$ and $\triangle_{\succ, 0} = \triangle_\succ$) such that $(C[[t]], \alpha, \triangle_{\prec, t}, \triangle_{\succ, t})$ is a multiplicative hom-dendriform coalgebra over $\mathbb{K}[[t]].$
 
 Therefore, we must have the following identities hold
\begin{align}
\sum_{i+j= n} (\triangle_{\prec, i} \otimes  \alpha) \circ \triangle_{\prec, j} =~& \sum_{i+j= n} (\alpha \otimes (\triangle_{\prec, i} + \triangle_{\succ, i})) \circ \triangle_{\prec, j}, \label{def-c-1}\\
\sum_{i+j= n} (\triangle_{\succ, i} \otimes \alpha) \circ \triangle_{\prec, j} =~& \sum_{i+j= n}(\alpha \otimes \triangle_{\prec, i} ) \circ \triangle_{\succ, j}, \label{def-c-2}\\
\sum_{i+j= n} ((\triangle_{\prec, i} + \triangle_{\succ, i} ) \otimes \alpha ) \circ \triangle_{\succ, j} =~&  (\alpha \otimes \triangle_{\succ, i} ) \circ \triangle_{\succ, j} \label{def-c-3}
\end{align}
for each $n \geq 0$. These equations are called deformation equations for the hom-dendriform coalgebra. To write these equations in a compact form, we use the following notations. Define $\triangle_i : \mathbb{K}[C_2] \otimes C \rightarrow C \otimes C$ by $\triangle_i ([1]; ~) = \triangle_{\prec, i}$ and $\triangle_{i} ([2]; ~)  = \triangle_{\succ, i}$, for $i \geq 0$.  Then the equations (\ref{def-c-1})-(\ref{def-c-3}) can be simply reads as
 \begin{align*}
 \sum_{i+j = n } \triangle_i \bullet_1 \triangle_j =  \sum_{i+j = n } \triangle_i \bullet_2 \triangle_j , ~~~ n \geq 0.
 \end{align*}
 The above identity holds automatically for $n = 0$ as we know $\triangle \bullet_1 \triangle =   \triangle \bullet_2 \triangle.$ For $n=1$, we get
 \begin{align*}
 \triangle \bullet_1 \triangle_1 + \triangle_1 \bullet_1 \triangle = \triangle \bullet_2 \triangle_1 + \triangle_1 \bullet_2 \triangle
 \end{align*}
 or equivalently, $\delta_{\alpha, \mathrm{coDend}} (\triangle_1) = 0.$
 Therefore, $\triangle_1 \in C^2_{\alpha, \mathrm{coDend}} (C, C)$ defines a $2$-cocycle. It is called the infinitesimal of the deformation. Thus, the infinitesimal of a deformation is a $2$-cocycle in the cohomology of the hom-dendriform coalgebra.
 
 \begin{defn}
Two deformations $(\triangle_{\prec, t}, \triangle_{\succ, t})$ and $(\triangle'_{\prec, t}, \triangle'_{\succ, t})$ of a hom-dendriform coalgebra $(C, \alpha, \triangle_\prec, \triangle_\succ)$ are said to be equivalent  if there exists a formal isomorphism $\Phi_t = \sum_{i \geq 0} t^i \Phi_i : C[[t]] \rightarrow C[[t]]$ with each $\Phi_i$ commute with $\alpha$ and $\Phi_0 = \text{id}_C$ such that
\begin{align*}
\triangle'_{\prec, t} \circ \Phi_t =~& (\Phi_t \otimes \Phi_t) \circ \triangle_{\prec, t}, \\
\triangle'_{\succ, t} \circ \Phi_t =~& (\Phi_t \otimes \Phi_t) \circ \triangle_{\succ, t}.
\end{align*}
\end{defn}
Note that each $\Phi_i : C \rightarrow C$ can be thought of as an element in $C^1_{\alpha, \text{coDend}}  (C, C)$ Then the above conditions  of the equivalence can be simply expressed as
\begin{align*}
\sum_{i+j=n} \triangle_i' ([r]; ~\_~) \circ \Phi_j = \sum_{i+j +k = n} ( \Phi_i \otimes \Phi_j ) \circ \triangle_k ([r]; ~\_~),
\end{align*}
for $n \geq 0$ and $[r] = [1], [2].$
The condition for $n = 0$ holds automatically as $\Phi_0 = \text{id}_C$. For $n = 1$, we have
\begin{align*}
\triangle_1' ([r]; ~\_~) ~+~ \triangle ([r]; ~\_~) \circ \Phi_1 =~ \triangle_1 ([r]; ~\_~) ~+ ~ (\text{id} \otimes \Phi_1 ) \circ \triangle ([r]; ~\_~) ~+~ (\Phi_1 \otimes \text{id}) \circ \triangle ([r]; ~\_~).
\end{align*}
This shows that the difference $\triangle_1' - \triangle_1$ is a coboundary $\delta_{\alpha, \text{coDend}} { (\Phi_1)}$. Therefore, the infinitesimals corresponding to equivalent deformations are cohomologous, hence, they gives rise to same cohomology class in $H^2_{\alpha, \text{coDend}} (C, C).$ 
 
A deformation $(\triangle_{\prec, t}, \triangle_{\succ, t})$ of a multiplicative hom-dendriform coalgebra $(C, \alpha, \triangle_\prec, \triangle_\succ)$ is said to be trivial if it is equivalent to the deformation $(\triangle'_{\prec, t} = \triangle_\prec, \triangle'_{\succ, t} = \triangle_\succ)$.
 
 \begin{thm}
 Let $(C, \alpha, \triangle_\prec, \triangle_\succ)$ be a multiplicative hom-dendriform coalgebra. If $H^2_{\alpha, \mathrm{coDend}} (C, C) = 0$ then every deformation of $C$ is equivalent to a trivial deformation.
 \end{thm}
 
 The proof of the above theorem is similar to Lemma \ref{non-triv-lemma} and Theorem \ref{h2}. Hence we omit the details. It follows that $H^2_{\alpha, \mathrm{coDend}} (C, C) = 0$ implies that the hom-dendriform coalgebra $C$ is rigid.
 
 Similarly, one may study an extension of a finite order deformation to the next order. The vanishing of the third cohomology ensures such an extension.
 
 \medskip
 
\noindent {\bf Acknowledgement.} The author would like to thank the referee for his/her comments on the earlier version of the manuscript.


\begin{thebibliography}{BFGM03}

\bibitem{aguiar} M. Aguiar, Pre-Poisson algebras, {\em Lett. Math. Phys.} 54 (2000), no. 4, 263-277.

\bibitem{amm-ej-makh}
F. Ammar, Z. Ejbehi and A. Makhlouf, Cohomology and deformations of Hom-algebras, 
{\em J. Lie Theory} 21 (2011), no. 4, 813-836.

\bibitem{dnr} S. D\u{a}sc\u{a}lescu, C. N\u{a}st\u{a}sescu and \c{S}. Raianu, Hopf algebras.
An introduction. Monographs and Textbooks in Pure and Applied Mathematics, 235. {\em Marcel Dekker, Inc., New York,} 2001.

\bibitem{das2}
A. Das, Homotopy $G$-algebra structure on the cochain complex of hom-type algebras, {\em C. R. Math. Acad. Sci. Paris} 356 (2018), no. 11-12, 1090-1099.

\bibitem{das3}
A. Das, Cohomology and deformations of dendriform algebras, and $\text{Dend}_\infty$-algebras, arXiv preprint, arXiv:1903.11802

\bibitem{das4}
A. Das, Cohomology and deformations of dendriform coalgebras, arXiv preprint, arXiv:1907.08255

\bibitem{jian}
R.-Q. Jian and J. Zhang, Rota-Baxter coalgebras, arXiv preprint, https://arxiv.org/abs/1409.3052

\bibitem{gers}
M. Gerstenhaber, On the deformation of rings and algebras,
{\em Ann. of Math.} (2) 79 (1964) 59-103. 

\bibitem{gers-voro}
M. Gerstenhaber and A. A. Voronov, Homotopy $G$-algebras and moduli space operad,
{\em Internat. Math. Res. Notices} 1995, no. 3, 141-153. 

\bibitem{hart-larson-sil} J. T. Hartwig, D. Larsson and S. D. Silvestrov, Deformations of Lie algebras using $\sigma$-derivations,
{\em J. Algebra} 295 (2006), no. 2, 314-361. 

\bibitem{loday}
J.-L. Loday, Dialgebra, {\em Dialgebras and related operads,} 7-66, 
Lecture Notes in Math., 1763, {\em Springer, Berlin,} 2001. 

\bibitem{lod-val} J.-L. Loday and B. Vallette, Algebraic operads, Grundlehren der Mathematischen Wissenschaften [Fundamental Principles of Mathematical Sciences], 346. {\em Springer, Heidelberg,} 2012.

\bibitem{liu-makh-men-pan}
L. Liu, A. Makhlouf, C. Menini and F. Panaite, Rota-Baxter operators on BiHom-associative algebras and related structures, to appear in {\em Colloq. Math.}, arXiv:1703.07275

\bibitem{makh-sil3}
A. Makhlouf and S. Silvestrov, Hom-algebras and Hom-coalgebras,
{\em J. Algebra Appl.} 9 (2010), no. 4, 553-589.

\bibitem{makh} A. Makhlouf, Hom-dendriform algebras and Rota-Baxter Hom-algebras, {\em Operads and universal algebra}, 147-171, 
Nankai Ser. Pure Appl. Math. Theoret. Phys., 9, {\em World Sci. Publ., Hackensack, NJ,} 2012.

\bibitem{makh-sil} A. Makhlouf and S. Silvestrov, Notes on $1$-parameter formal deformations of Hom-associative and Hom-Lie algebras,
{\em Forum Math.} 22 (2010), no. 4, 715-739. 

\bibitem{ma-zheng} T. Ma and H. Zheng, Some results on Rota-Baxter monoidal Hom-algebras, {\em Results Math.} 72 (2017), no. 1-2, 145-170. 

\bibitem{radford} D. E. Radford, Hopf algebras.
Series on Knots and Everything, 49. {\em World Scientific Publishing Co. Pte. Ltd., Hackensack, NJ,} 2012.

\end{thebibliography}
\end{document}